\documentclass[12pt]{amsart}
\usepackage{MnSymbol}
\usepackage{mathrsfs}
\usepackage[utf8]{inputenc}
\usepackage{amsmath,amsthm}

\DeclareMathAlphabet\mathbb{U}{msb}{m}{n}

\usepackage[all]{xy}

\numberwithin{equation}{section}

\def\xyma{\xymatrix@M.7em}

\numberwithin{equation}{section}

\usepackage[T2A]{fontenc}
\newtheorem{cor}{Corollary}[section]

\newtheorem{prop}{Proposition}[section]

\newtheorem{lemma}{Lemma}[section]

\theoremstyle{definition}

\newcommand{\mono}{\rightarrowtail}
\newcommand{\epi}{\twoheadrightarrow}
\def\LG{{\mathcal{LG}}}
\def\ZZ{{\mathbb Z}}
\def\QQ{{\mathbb Q}}

\begin{document}

\title{A finite $\mathbb Q$-bad space}
\author{Sergei O. Ivanov}
\address{Chebyshev Laboratory, St. Petersburg State University, 14th Line, 29b,
Saint Petersburg, 199178 Russia} \email{ivanov.s.o.1986@gmail.com}

\author{Roman Mikhailov}
\address{Chebyshev Laboratory, St. Petersburg State University, 14th Line, 29b,
Saint Petersburg, 199178 Russia and St. Petersburg Department of
Steklov Mathematical Institute} \email{rmikhailov@mail.ru}

\begin{abstract} We prove that, for a free noncyclic
group $F$, the second homology group $H_2(\hat F_\mathbb Q, \mathbb Q)$ is an uncountable $\mathbb Q$-vector
space, where $\hat F_{\mathbb Q}$ denotes the $\mathbb Q$-completion of $F$. This solves a problem of
A.~K.~Bousfield for the case of rational coefficients. As a direct consequence of
this result, it follows that a wedge of two or more circles is $\mathbb Q$-bad in the sense
of Bousfield–Kan. The same methods as used in the proof of the above result
serve to show that $H_2(\hat F_{\ZZ}, \ZZ)$ is not a divisible group, where $\hat F_{\ZZ}$ is the integral
pronilpotent completion of $F.$
\end{abstract}
\maketitle

\section{Introduction}

In the foundational work \cite{BK}, A. K. Bousfield and D. M. Kan 
introduced the concept of $R$-completion of a space for a
commutative ring $R$. For a space $X$, there is an $R$-completion functor $X\mapsto R_\infty X$ such that a map between two
spaces $f: X\to Y$ induces an isomorphism of reduced homology
$\tilde H_*(X, R)\cong \tilde H_*(Y, R)$ if and only if it induces a
homotopy equivalence $R_\infty X\simeq R_\infty Y$.  Thus, $R$-completion can be viewed as an approximation of the $R$-homology localization of a space, defined in \cite{Bousfield75}. For certain classes of spaces, such as nilpotent spaces, $R$-completion and $R$-homology localization coincide.

The $R$-completion functor for spaces is closely related to the $R$-completion functor for groups. For a group $G$, denote by $\{\gamma_i(G)\}_{i\geq 1}$ the lower central series of $G$. We will consider the pronilpotent completion $\hat G_\ZZ$ of $G$ as well as the $\QQ$-completion $\hat G_\QQ$ defined as
$$\hat G_\ZZ=\varprojlim G/\gamma_i(G), \hspace{1cm} \hat G_\QQ=  \varprojlim\ G/\gamma_i(G) \otimes \QQ. $$
Here $G/\gamma_i(G)\otimes \mathbb Q$ is the Malcev $\mathbb Q$-localization of the nilpotent group $G/\gamma_i(G)$. One can find the definition of $\mathbb Z/p$-completion $\hat G_{\mathbb Z/p}$ in \cite{BK}, \cite{Bousfield77}. In this paper we do not use $\ZZ/p$-completion and work only over $\mathbb Z$ or $\mathbb Q$. It is shown in \cite[Ch.4]{BK} that $R$-completion of a connected space $X$ can be constructed explicitly as $\bar W\widehat{(GX)}_R,$ where $G$ is the Kan loop simplicial group, $\widehat{(GX)}_R$ is the $R$-completion of $GX$ and $\bar W$ is the classifying space functor.

A space $X$ is called $R$-good if the map $X\to R_\infty X$
induces an isomorphism of reduced homology $\tilde
H_*(X,R)\cong \tilde H_*(R_\infty X,R)$, and called $R$-bad
otherwise.  In other words, for $R$-good spaces $R$-homology localization and $R$-completion coincide.

There are a lot of examples of $R$-good and $R$-bad spaces. The key example of \cite{BK} is the projective plane $\mathbb RP^2$, which is $\mathbb Z$-bad. This fact implies that some finite wedge of circles is also $\mathbb Z$-bad. It is shown in \cite{Bousfield77} that a wedge of two circles is $\mathbb Z$-bad. In \cite{Bousfield92}, Bousfield proved that, for any prime $p$, a wedge of circles is $\mathbb Z/p$-bad, thus providing first example of a finite $\mathbb Z/p$-bad space.
For $R$ a subring of the rationals or $\mathbb Z/n,\ n\geq 2$, and a free group $F$, there is a weak equivalence (\cite[5.3]{BK})
$$
R_\infty K(F,1)\simeq K(\hat F_R,1).
$$
Therefore, the question of $R$-goodness of a wedge of circles is reduced to the question of nontriviality of the higher $R$-homology of the $R$-completion of a free group. The same question naturally appears in the theory of $HR$-localizations of groups. In \cite[Problem 4.11]{Bousfield77}, Bousfield posed the following problem:\\ \\
{\bf Problem.}\ {\em (Bousfield)}  Does $H_2(\hat F_R, R)$ vanish when $F$ is a finitely generated free group and $R=\mathbb Q$ or $R=\mathbb Z/n$?\\

\noindent In the recent paper \cite{IvanovMikhailovPreprint_pro-p}, the authors show that, for $R=\mathbb Z/n,$ $H_2(\hat F_R,R)$ is an uncountable group, solving the above problem for the case $R=\mathbb Z/n$. The key step in \cite{IvanovMikhailovPreprint_pro-p} substantially uses the theory of profinite groups. Hence the method given in \cite{IvanovMikhailovPreprint_pro-p} cannot be directly transferred to the case $R=\mathbb Q$.

In this paper we answer Bousfield's problem over $\mathbb Q$. Our main results are the following theorems.

\vspace{.5cm}\noindent{\bf Theorem 1.} For a finitely generated noncyclic free group $F$, $H_2(\hat F_\mathbb Q, \mathbb Q)$ is uncountable.

\

\noindent Moreover, we prove that the image of the map $H_2(\hat F_\ZZ,\ZZ)\to H_2(\hat F_\QQ,\QQ) $ is uncountable.

\vspace{.5cm}\noindent{\bf Theorem 2.} For a finitely generated noncyclic free group $F$ and a prime $p$, $H_2(\hat F_{\ZZ}, \mathbb Z/p)$ is uncountable. In particular, $H_2(\hat F_{\ZZ},\ZZ)$ is not divisible.

\vspace{.5cm}\noindent Theorem 2 answers a problem posted in \cite{IvanovMikhailov16}. As mentioned above, $\mathbb Q_\infty K(F,1)=K(\hat F_{\mathbb Q},1).$ Therefore, Theorem 1 implies the following:

\vspace{.5cm}\noindent{\bf Corollary.} A wedge of $\geq 2$ circles is $\mathbb Q$-bad.

\vspace{.5cm} \noindent As far as known to the authors, this is the first known example of a finite $\mathbb Q$-bad space. 

The proof is organized as follows. In Section 2 we discuss technical results about power series. The main result of Section 2, Proposition 2.1, states that the kernel of the natural map between a rational power series ring and the coinvariants of the diagonal action of the rationals on the exterior square $\mathbb Q[\![x]\!] \to \Lambda^2(\mathbb Q[\![x]\!])_{\mathbb Q},$ given by $f\mapsto f\wedge 1,$ is countable. (In the proof of the proposition we  use the fact that the group algebra $\QQ[\QQ]$ is countable. In the similar statement for the $\ZZ/p$-completion we should consider the mod-$p$ group algebra of the group of $p$-adic integers $\ZZ/p[\ZZ_p],$ which is uncountable. So this method fails for $\ZZ/p$-completions.) Here In Section 3, we consider the integral lamplighter group:
$$\LG=\langle a,b \mid [a,a^{b^i}]=1, \ i\in \mathbb Z \rangle,$$
which is isomorphic to the wreath product of two infinite cyclic groups, as well as its $p$-analog $\mathbb Z/p\wr C$, where $C$ denotes an infinite cyclic group. The group $\LG$ is metabelian; therefore, its completions $\widehat{\LG}_\mathbb Z$ and $\widehat{\LG}_\mathbb Q$ can be easily described (see \eqref{eq_LG_com1}, \eqref{eq_LG_com2}), and the homology group $H_2(\widehat{\LG}_\mathbb Q, \mathbb Q)$ is isomorphic to the natural coinvariant quotient of the exterior square $\Lambda^2(\mathbb Q[[x]])$. The key step in the proof of the main results ocuurs in Section 4, in Proposition 4.1. Let $F=F(a,b)$ be a free group of rank two with generators $a,b$. We construct (see Proposition 4.1) an uncountable collection of elements $r_q, s_q\in \hat F_\mathbb Z$ such that $[r_q,a][s_q,b]=1$ in $\hat F_\mathbb Z$. One can consider the group homology $H_2(\hat F_\mathbb Z,\mathbb Z)$ as a kernel of the commutator map $\hat F_\mathbb Z\wedge \hat F_\mathbb Z\to \hat F_\mathbb Z, a\wedge b \mapsto [a,b] $ where $\hat F_\mathbb Z\wedge \hat F_\mathbb Z$ is the non-abelian exterior square of $\hat F_\mathbb Z$ \cite{BrownLoday}. Therefore, the pairs of elements $(r_q\wedge a)(s_q\wedge b)\in \hat F_\mathbb Z\wedge \hat F_\mathbb Z$ define certain elements of $H_2(\hat F_\mathbb Z,\mathbb Z)$. Next we consider the following natural maps between homology groups of different completions, which are induced by the standard projection $F\to \LG$:
$$
\xyma{H_2(\hat F_\mathbb Q,\mathbb Q)\ar@{->}[d] & H_2(\hat F_\mathbb Z,\mathbb Z)\ar@{->}[r]\ar@{->}[rd]\ar@{->}[ld] \ar@{->}[l] & H_2(\hat F_\mathbb Z, \mathbb Z/p)\ar@{->}[d]\\
H_2(\widehat{\LG}_\mathbb Q,\mathbb Q) & & H_2(\widehat{\LG}_\mathbb Z, \mathbb Z/p), }
$$
and show, in the final Section 5, that the sets of  images of the elements $(r_q\wedge a)(s_q\wedge b)$ in $H_2(\widehat{\LG}_\mathbb Q,\mathbb Q)$ and $H_2(\widehat{\LG}_\mathbb Z, \mathbb Z/p)$ are uncountable. Theorems 1 and 2 follow.

\section{Technical results about power series}

We denote by $C$ an infinite cyclic group written multiplicatively as $C=\langle t\rangle.$ For a commutative ring $R$ we denote by $R[\![x]\!]$ the ring of formal power series over $R$ and by $R[C]$ the group algebra of $C$. Consider the multiplicative homomorphism
$$\tau:C \longrightarrow R[\![x]\!], \hspace{1cm} \tau(t)=1+x.$$
The induced ring homomorphism is denoted by the same letter
$$ \tau:R[C] \longrightarrow R[\![x]\!].$$

\begin{lemma}\label{lemma_powerseries}
If we denote by $I$ the augmentation ideal of $R[C]$ and set $R[C]^\wedge=\varprojlim R[C]/I^i,$ then $\tau(I^n)\subseteq x^n\cdot R[\![x]\!]$ and $\tau$ induces isomorphisms $$R[C]/I^n\cong R[x]/x^n$$
$$R[C]^\wedge\cong R[\![x]\!].$$
\end{lemma}
\begin{proof}
If we set $x=t-1,$ we obtain $R[C]=R[x,(1+x)^{-1}]$ and $I=x\cdot R[C].$ Observe that the image of the element $1+x$ in $R[x]/x^n$ is invertible. Since  localization at the element $1+x$ is an exact functor, the short exact sequence $x^n\cdot R[x] \mono R[x]\epi R[x]/x^n$ gives the short exact sequence $(x^n\cdot R[x])_{1+x} \mono R[C]\epi R[x]/x^n.$ It follows that $R[C]/x^n\cong R[x]/x^n.$ The assertion follows.
\end{proof}

Denote by $\sigma$ the antipode  of the group ring $R[C]:$
$$\sigma:R[C]\longrightarrow R[C],\hspace{1cm}\sigma(\sum a_it^i)=\sum a_it^{-i}.$$
Obviously $\sigma(I^n)=I^n,$ and hence it induces a continuous involution
$$\hat \sigma: R[C]^\wedge \longrightarrow R[C]^\wedge.$$
Composing this involution with the isomorphism $R[C]^\wedge\cong R[\![x]\!]$ we obtain
a continuous involution
$$\tilde \sigma: R[\![x]\!] \longrightarrow R[\![x]\!]$$
such that
$$\tilde \sigma(x)=-x+x^2-x^3+ x^4 -\dots\ .$$

Consider the case $R=\QQ.$ Note that the set $1+x\cdot \QQ[\![x]\!]$ is a group and  there is a unique way to define $r$-power map $f\mapsto f^r$ for $r\in \QQ$ that extends  the usual power map $f\mapsto f^n$ such that $f^{r_1r_2}=(f^{r_1})^{r_2}$ (see Lemma 4.4 of \cite{IvanovMikhailov16}). This map is defined by the formula
$$f^r=\sum_{n=0}^\infty \binom{r}{n} (f-1)^n,$$
where $\binom{r}{n}=r(r-1)\dots (r-n+1)/n!.$  Denote by $C\otimes \QQ$  the group $\QQ$ written multiplicatively as powers of $t$: $C\otimes \QQ=\{t^r\mid r\in \QQ\}.$ Consider the multiplicative homomorphism
\begin{equation}\label{eq_tau}
 \tau_\QQ:C\otimes \QQ \longrightarrow \QQ[\![ x ]\!]
\end{equation}
that extends $\tau:C\to \QQ[\![x]\!]:$
$$ \tau_\QQ(t^r)=(1+x)^r.$$
The induced ring homomorphism is denoted by the same letter
$$ \tau_\QQ:\QQ[C\otimes \QQ] \longrightarrow \QQ[\![x]\!].$$
This homomorphism allows us to consider $\QQ[\![x]\!]$ as a $\QQ[C\otimes \QQ]$-module.
We claim that the homomorphism $\tau_{\mathbb Q}:\QQ[C\otimes \QQ]\to \QQ[\![x]\!]$ respects the involutions:
\begin{equation}\label{eq_resp_involutions}
\tau_{\QQ} \circ \sigma_{C\otimes \QQ} = \tilde \sigma \circ \tau_{\QQ},
\end{equation}
where $\sigma_{C\otimes \QQ}$ is the antipode on $\mathbb Q[C\otimes \QQ].$ Indeed, we have that $(1+x)^{-1}=\tilde \sigma(1+x)=\tilde \sigma((1+x)^{1/n})^n$ and then $\tilde \sigma((1+x)^{1/n})=(1+x)^{-1/n},$ which implies $\tilde \sigma((1+x)^{r})=(1+x)^{-r}$ for any $r\in \QQ,$ and hence $\tau_{\QQ} ( \sigma_{C\otimes \QQ} (t^r))= \tilde \sigma ( \tau_{\QQ}(t^r) )$ for any $r\in \QQ.$

\begin{prop} \label{prop_kernel_to_exterior}
\begin{enumerate}
\item Denote by $\Lambda^2 (\QQ[\![x]\!]) $ the exterior square of $\QQ[\![x]\!]$ considered as a $C\otimes \QQ$-module with the diagonal action. Consider the space of $C\otimes \QQ$-coinvariants $(\Lambda^2 (\QQ[\![x]\!]) )_{C\otimes \QQ}$. Then the kernel of the homomorphism   $$\theta_\QQ: \QQ[\![x]\!] \longrightarrow (\Lambda^2 (\QQ[\![x]\!]))_{C\otimes \QQ}$$
$$\theta_\QQ(f)= f \wedge 1$$  is countable.

\item Let $p$ be a prime. Denote by $\Lambda^2 (\ZZ/p[\![x]\!]) $ the exterior square of $\ZZ/p[\![x]\!]$ considered as a $C$-module with the diagonal action. Consider the space of $C$-coinvariants $(\Lambda^2 (\ZZ/p[\![x]\!]) )_{C}$. Then the kernel of the homomorphism  $$\theta_{\ZZ/p}:\ZZ/p[\![x]\!] \longrightarrow (\Lambda^2 (\ZZ/p[\![x]\!]))_{C}$$
$$\theta_{\ZZ/p}(f)= f \wedge 1$$  is countable.
\end{enumerate}
\end{prop}
\begin{proof}(1) Consider the linear map
$$\alpha: \Lambda^2(\QQ[\![x]\!]) \longrightarrow \QQ[\![x]\!]^{\otimes 2}, \hspace{1cm} \alpha(f\wedge g)=f\otimes g - g\otimes f. $$ 
Note that this is a homomorphism of $\mathbb Q[C\otimes \mathbb Q]$-modules, where the action of $C\otimes \mathbb Q$ is defined diagonally in both cases. Hence, it induces a linear map:
$$\alpha_{C\otimes Q}: (\Lambda^2(\QQ[\![x]\!]))_{C\otimes \mathbb Q} \longrightarrow (\QQ[\![x]\!]^{\otimes 2})_{C\otimes \mathbb Q}.$$ 
Next, we consider the homomorphism
$$\beta: (\QQ[\![x]\!]^{\otimes 2})_{C\otimes \mathbb Q} \longrightarrow \QQ[\![x]\!] \otimes_{\mathbb Q[C\otimes \mathbb Q] } \QQ[\![x]\!], \hspace{1cm} \beta(f\otimes g)=f\otimes \tilde \sigma(g), $$
which is well defined because $\tau_\QQ$ respects the involutions \eqref{eq_resp_involutions}: $ ft^r\otimes \tilde\sigma(gt^r)=ft^r\otimes \tilde\sigma(g) t^{-r}=f\otimes \tilde \sigma(g).$ Denote by $K$ the subfield of the field of Laurent power series $\QQ(\!(x)\!)$ generated by the image of $\tau_\QQ.$  Then there is a map 
$$ \gamma:\QQ[\![x]\!] \otimes_{\mathbb Q[C\otimes \mathbb Q] } \QQ[\![x]\!] \longrightarrow \QQ(\!(x)\!) \otimes_{K } \QQ(\!(x)\!).$$ 
The composition $$\gamma \circ \beta \circ \alpha_{C\otimes \QQ} \circ \theta_{\QQ}: \QQ[\![x]\!] \to \QQ(\!(x)\!) \otimes_{K } \QQ(\!(x)\!)$$ sends $f$ to $f\otimes 1 - 1\otimes \tilde \sigma(f).$ Note that for any vector spaces $V,U$ over any field and any elements $v_1,v_2\in V$ and $u_1,u_2\in U,$ if $v_1$ and $v_2$ are linearly independent and $u_1\ne 0,u_2\ne 0,$ then $v_1\otimes u_1$ and $v_2\otimes u_2$ are linearly independent in $V\otimes U.$  It follows that for any $f\in \QQ[\![x]\!]\setminus K$ we have that $f\otimes 1 - 1 \otimes \tilde\sigma(f)\ne 0$ in $\QQ(\!(x)\!) \otimes_{K } \QQ(\!(x)\!).$  Therefore ${\rm Ker}(\theta_\QQ)\subseteq K.$ Since the fraction filed of the countable algebra $\QQ[C\otimes \QQ]$ is countable, $K$ is countable. The assertion follows.

(2) The proof is the same.
\end{proof}

\section{Completions of lamplighter groups $\LG$ and $\LG(p)$}

Recall the definition of the tensor square for a non-abelian group \cite{BrownLoday}.
For a group $G$, the tensor square $G\otimes G$ is the group generated by the symbols
$g\otimes h,\ g,h\in G,$ satisfying the following defining relations:
\begin{align*}
& fg\otimes h=(g^{f^{-1}}\otimes h^{f^{-1}})(f\otimes h),\\
& f\otimes gh=(f\otimes g)(f^{g^{-1}}\otimes h^{g^{-1}}),
\end{align*}
for all $f,g,h\in G$. The exterior square $G\wedge G$ is defined as
$$
G\wedge G:=G\otimes G/\langle g\otimes g,\ g\in G\rangle.
$$
The images of the elements $g\otimes h$ in $G\wedge G$  will be denoted by $g\wedge h$.
If $G=E/R$ for a free group $E$, there is a natural isomorphism $G\wedge G\cong \frac{[E,E]}{[R,E]}.$

For any group $G$, there is a natural short exact sequence
$$0 \longrightarrow H_2(G,\ZZ) \longrightarrow G \wedge G \overset{[-,-]}\longrightarrow [G,G] \longrightarrow 1
$$
(see \cite[(2.8)]{BrownLoday} and \cite{Miller}).
Let  $g_1,\dots,g_n, h_1,\dots,h_n\in G$ be elements such that $ [g_1, h_1] \dots [g_n,h_n]=1.$ Then  the element $(g_1\wedge h_1) \dots (g_n \wedge h_n) $  defines an element in $H_2(G,\ZZ):$
$$(g_1\wedge h_1) \dots (g_n \wedge h_n) \in H_2(G,\ZZ).$$
If $R$ is a commutative ring, then the image of $(g_1\wedge h_1) \dots (g_n \wedge h_n)$ in $H_2(G,R)$ is denoted by
$$((g_1\wedge h_1) \dots (g_n \wedge h_n))\otimes R \in H_2(G,R).$$

\vspace{.25cm}

We will consider two versions of the lamplighter group. The integral lamplighter group $$\LG=\mathbb Z \wr C=\langle a,b \mid [a,a^{b^i}]=1, \ i\in \mathbb Z \rangle$$ and
the $p$-lamplighter group for a prime $p$ $$\LG(p)=\mathbb Z/p \wr C=\langle a,b \mid [a,a^{b^i}]=a^p=1,  \ i\in \mathbb Z \rangle.$$
Observe that $\LG=\mathbb Z[C]\rtimes C$ and $\LG(p)=\ZZ/p[C] \rtimes C.$ Using Lemma \ref{lemma_powerseries} and \cite[Prop. 4.7]{IvanovMikhailov16}, we obtain
\begin{equation}\label{eq_LG_com1}
\widehat{\LG}_\ZZ=\ZZ[\![x]\!]\rtimes C
\end{equation}
and
\begin{equation}\label{eq_LG_com2}
\widehat{\LG}_\QQ=\QQ[\![x]\!]\rtimes (C\otimes \QQ), \hspace{1cm} \widehat{\LG(p)}_\ZZ=\ZZ/p[\![x]\!] \rtimes C,
\end{equation}
where $C$ acts on $\ZZ[\![x]\!]$ and $\ZZ/p[\![x]\!]$ via $\tau$ and $C\otimes \QQ$ acts on $\QQ[\![x]\!]$ via $\tau_\QQ$.

\begin{prop}\label{prop_homology_of_lamplighter} There are  isomorphisms
$$(\Lambda^2(\QQ[\![x]\!]))_{C\otimes \QQ}\cong  H_2(\widehat{\LG}_\QQ,\QQ), $$
$$(\Lambda^2(\ZZ/p[\![x]\!]))_{C} \cong H_2(\widehat{\LG(p)}_\ZZ,\ZZ/p) $$
in both cases given by
$$f \wedge f'  \mapsto ((f,1)\wedge (f',1)) \otimes R ,$$
where $R=\QQ$ and $R=\ZZ/p$ respectively. 
\end{prop}
\begin{proof}
Consider the short exact sequence $\QQ[\![x]\!] \mono \widehat{\LG}_\QQ\epi (C\otimes \QQ)$ and the associated spectral sequence $E.$ Since $\QQ=\varinjlim \frac{1}{n!}\ZZ$ and homology commutes with direct limits, we have $H_n(C\otimes \QQ, -)=0$ for $n\geq 2$. It follows that $E^2_{i,j}=0 $ for $i\geq 2$ and hence there is a short exact sequence
$$0\longrightarrow E^2_{0,2} \longrightarrow H_2(\widehat{\LG}_\QQ,\QQ) \longrightarrow E^2_{1,1} \longrightarrow 0.$$
Observe that the action of $C$ on $\QQ[\![x]\!]$ has no invariants. Then
 $$E^2_{1,1}=H_1(C\otimes \QQ, \QQ[\![x]\!])=\varinjlim H_1(C\otimes \frac{1}{n!}\ZZ, \QQ[\![x]\!])=\varinjlim \QQ[\![x]\!]^{C\otimes \frac{1}{n!}\ZZ}=0.$$ It follows that the map
\begin{equation}\label{eq_mapE_2}
H_2(\QQ[\![x]\!],\QQ)_{C\otimes \QQ}=E_{0,2}^2\longrightarrow H_2(\widehat{\LG}_\QQ,\QQ)
\end{equation}
 is an isomorphism. The map is induced by the map $\QQ[\![x]\!]\mono \widehat\LG_\QQ$ that sends $f\in \QQ[\![x]\!]$ to $(f,1)\in  \widehat\LG_\QQ.$ Then the isomorphism \eqref{eq_mapE_2} sends $f \wedge f'$ to $((f,1)\wedge (f',1)) \otimes \QQ.$ Using the isomorphism $\Lambda^2(\QQ[\![x]\!])\cong H_2(\QQ[\![x]\!],\QQ)$ we obtain the assertion.

 The second isomorphism can be proved similarly.
\end{proof}

\section{Completion of a free group}

For elements of groups or Lie rings, we will use the left-normalized notation $[a_1,\dots,a_n]:=[[a_1,\dots,a_{n-1}],a_n]$ and the following notation for Engel commutators
$$[a,_0b]:=a,\hspace{1cm} [a,_{i+1}b]=[[a,_ib],b]$$
for $i\geq 0.$

For all elements $a,b$ of a Lie ring, the Jacobi identity implies that
$$
[a,b,a,b]+[b,[a,b],a]+[[a,b],[a,b]]=0.
$$
It follows that
\begin{equation}\label{lieident1}
[a,b,b,a]=[a,b,a,b].
\end{equation}
The following lemma is a generalization of this identity.

\begin{lemma} Let $L$ be a Lie ring, $a,b\in L$ and $n\geq 1.$ Then
\begin{equation}\label{lieident}
[[a,_{2n}b],a]=\left[\sum_{i=0}^{n-1}(-1)^{i}[[a,_{2n-1-i}b],[a,_{i}b]]\ ,\  b\ \right].
\end{equation}
\end{lemma}
\begin{proof}
The Jacobi identity implies that
\begin{equation}
 [[a,_{2n-i}b], [a,_{i} b]]+[[a,_{2n-1-i},b],[a,_{i+1}b]]= [[a,_{2n-1-i}b],[a,_{i} b],b]
\end{equation}
for $0\leq i\leq n-1.$ Taking the alternating sum of these identities and using the fact that $[[a,_nb],[a,_nb]]=0,$ we obtain the assertion.
\end{proof}

\begin{cor}\label{corident} Let $F=F(a,b)$ be a free group with generators $a,b$.  For any $n\geq 1$,
$$[[a,_{2n}b],a] \equiv
\left[\prod_{i=0}^{n-1} [[a,_{2n-1-i}b],[a,_{i}b]]^{(-1)^i}\ ,\ b\ \right]\mod \gamma_{2n+3}(F).
$$
\end{cor}

We denote by $F$ the free group on two variables $F=F(a,b)$ and denote by  $\varphi:F\to \LG$ the obvious epimorphism to the integral lamplighter group. It induces a homomorphism between pronilpotent completions
$$\hat \varphi:\hat F_\ZZ \to \widehat{\LG}_\ZZ.$$ Note that $$\varphi([u,v])=1 \text{ for }u, v \in \langle a \rangle^F, $$
where $\langle a \rangle^F $ is the normal subgroup of $F$ generated by $a.$

\begin{prop}\label{proposition_main} For any sequence of integers $q=(q_1,q_2,\dots),$ there exists a pair of elements $r_q,s_q\in \gamma_3(\hat F_\ZZ)$ such that
\begin{enumerate}
\item  $[r_q,a][s_q,b]=1$;
\item $\hat \varphi (s_q)=1;$
\item
$ \hat \varphi(r_q)=\prod_{i=3}^{\infty} [a,_{i-1} b]^{n_i},
$
where $n_{2i+1}=q_i$ for $i\geq 1$ and $n_{2i}$ are some integers\\ (we control only odd terms of the product).
\end{enumerate}
\end{prop}

\begin{proof}
We claim that there exist sequences of  elements
$
r_q^{(3)}, r_q^{(4)}, \dots \in F$ and  $ s_q^{(3)}, s_q^{(4)}, \dots \in F
$
such that
\begin{enumerate}
\setcounter{enumi}{-1}
\item  $r_q^{(k)},s_q^{(k)}\in \gamma_k(F);$
\item $[\prod_{i=3}^kr_q^{(i)}, a][\prod_{i=3}^ks_q^{(i)},b]\in \gamma_{k+2}(F);$
\item  $\varphi(s_q^{(k)})=1;$
\item  $\varphi( \prod_{i=3}^{k} r^{(i)}_q )\equiv  \prod_{i=3}^k [a,_{i-1} b]^{n_i} \mod \gamma_{k+1}(\LG),$ where $n_{2i+1}=q_i$ for $2i+1\leq k.$
\end{enumerate}
Then we take $r_q=\prod_{i=3}^\infty r_q^{(i)}$ and $s_q=\prod_{i=3}^\infty s_q^{(i)}$ and the assertion follows. So it is sufficient to construct such elements $r_q^{(k)}, s_q^{(k)}$ inductively.

In order to prove the base case we set
$$
r_q^{(3)}:=[a,b,b]^{q_1},\ s_q^{(3)}:=[a,b,a]^{-q_1}.
$$
Corollary \ref{corident}  with $n=1,$ implies that
$$
[r_q^{(3)},a][s_q^{(3)},b]\in \gamma_5(F).
$$
Clearly $s_q^{(3)},r_q^{(3)}\in\gamma_3(F),$  $\varphi(s_q^{(3)})=1$ and $\varphi(r_q^{(3)})=[a,_2b]^{q_1}.$

In order to prove the inductive step, assume that we already constructed 
$$
r_q^{(3)},\dots, r_q^{(k)},\ s_q^{(3)},\dots, s_q^{(k)},
$$ with the properties (0)-(3). Construct $r_q^{(k+1)}$ and $s_q^{(k+1)}.$ Note that any element of $\gamma_{k+2}(F)/\gamma_{k+3}(F)$ can be presented as $[A,a][B,b] \cdot \gamma_{k+3}(F),$  where $A,B\in \gamma_{k+1}(F).$ Then
\begin{equation}\label{eq2222}
[\prod_{i=3}^kr_q^{(i)}, a][\prod_{i=3}^ks_q^{(i)},b]\equiv [A,a][B,b]\mod \gamma_{k+3}(F).
\end{equation}
Using that the images of $[A^{-1},a],[B^{-1},b]$ are in the center of $F/\gamma_{k+3}(F),$ that $\prod_{i=3}^kr_q^{(i)},  \prod_{i=3}^ks_q^{(i)} \in \gamma_3(F)$ and the identity $[xy,z]=[x,z]^y\cdot [y,z]$ we obtain
\begin{equation}\label{eq2222'}
[ \prod_{i=3}^kr_q^{(i)}A^{-1}, a]\ \cdot \ [\prod_{i=3}^ks_q^{(i)}B^{-1},b]\in  \gamma_{k+3}(F).
\end{equation}

Next we prove that $$\varphi(B)=1.$$
Since $B\in \gamma_{k+1}(F)$ we have
$$
B \equiv [a,_kb]^{e}c \mod \gamma_{k+2}(F),$$
where $ e\in \mathbb Z$ and $c$ is a product of powers of other basic commutators of weight $k+1.$ All these other basic commutators contain at least twice  $a.$ It follows that $\varphi(c)=1.$ Since $A\in \gamma_3(F) \subseteq \langle a \rangle^F$, we have $\varphi([A,a])=1.$ Moreover, $ \varphi( [\prod_{i=3}^kr_q^{(i)}, a][\prod_{i=3}^ks_q^{(i)},b])=1.$ Then
$$
[a,_{k+1}b]^e\in \gamma_{k+3}(\LG).
$$
This implies that $e=0$ and hence $\varphi(B)=1.$

If $k$ is odd, we do need to care about (3) and we just take $$r_q^{(k+1)}=A^{-1},\hspace{1cm} s_q^{(k+1)}=B^{-1}.$$ Indeed, it is easy to check that the properties (0)-(2) are satisfied and the property (3) automatically follows.

Suppose now that $k$ is even, say $k=2k'$. Consider the image of the element $\prod_{i=3}^k r_q^{(i)} \cdot A^{-1}$ in the quotient $\LG/\gamma_{k+2}(\LG)$.
By the induction hypothesis,
$$\varphi (\prod_{i=3}^k r_q^{(i)}) \equiv \prod_{i=3}^k [a,_{i-1} b]^{n_i} \cdot c' \mod \gamma_{k+2}(\LG), $$
where $c'\in \gamma_{k+1}(\LG)$.
Since the quotient
$\gamma_{k+1}(\LG)/\gamma_{k+2}(\LG)$
 is cyclic with generator
 $[a,_kb]\cdot \gamma_{k+2}(\LG)$,
 $$c'\equiv  [a,_kb]^y \mod \gamma_{k+2}(\LG)$$
 for some $y\in \mathbb Z$. For $n\geq 1,$ denote
$$
z_n:=\prod_{i=0}^{n-1} [[a,_{2n-1-i}b],[a,_{i}b]]^{(-1)^i}.
$$
Corollary \ref{corident} implies that
$$
[[a,_kb],a][z_{k'}^{-1},b]\in \gamma_{k+3}(F).
$$
We set $$r^{(k+1)}_q:=A^{-1}[a,_kb]^{q_{k'}-e},\ s^{(k+1)}_q:=B^{-1}z_{k'}^{-({q_{k'}-e})}.$$
Now
$$
[\prod_{i=3}^{k+1}r_q^{(i)}, a][\prod_{i=3}^{k+1}s_q^{(i)},b]\in \gamma_{k+3}(F)
$$
and
$$\varphi(\prod_{i=3}^{k+1}r_q^{(i)})\equiv \prod_{i=3}^{k+1} [a,_{i-1} b]^{n_i}. $$ The properties (0) and (2) are obvious.
\end{proof}

\section{Proof of Theorems 1 and 2}
Let $F$ be a free group of rank $\ \geq 2$ and $p$ be a prime. We will show that the image of the homomorphism
$H_2(\hat F_\ZZ,\ZZ) \to H_2(\hat F_\QQ,\QQ)$
is uncountable. The proof that the image of the map $H_2(\hat F_\ZZ,\ZZ) \to H_2(\hat F_\ZZ,\ZZ/p)$ is uncountable is similar.

Since the free group with two generators is a retract of a free group of higher rank, it is enough to prove this only for $F=F(a,b).$ The map
\begin{equation}\label{eq_map1}
H_2(\hat F_\ZZ,\ZZ) \to H_2(\widehat{\LG}_\QQ, \QQ)
\end{equation}
factors through $H_2(\hat F_\QQ,\QQ)$. Then it is enough to prove that the  image the map \eqref{eq_map1} is uncountable.

For $q\in \{ 0,1\}^{\mathbb N}$ we denote by $r_q,s_q$ some fixed elements of $\hat F_\ZZ$ satisfying properties (1), (2), (3) of Proposition \ref{proposition_main}. Then
$$\hat \varphi(r_q)=\prod_{i=3}^\infty [a,_{i-1} b]^{n_i(q)},$$
where $n(q)_{2i+1}=q_i$ and
$$[r_q,a][s_q,b]=1, \hspace{1cm} \hat \varphi(s_q)=1.$$
Set $$f_q= \sum_{i=3}^\infty  n_i(q) x^{i-1} \in \ZZ[\![x]\!].$$
 If we consider $\widehat{\LG}_\ZZ$ as the semidirect product $\ZZ[\![x]\!] \rtimes C,$ we obtain that $[a,_{i-1} b]=(x^{i-1},1)$ and hence
$$ \hat \varphi(r_q)=(f_q,1).  $$
If we denote by $\hat \varphi_\QQ$ the composition of $\hat \varphi$ with the map $\widehat{\LG}_\ZZ\to \widehat{\LG}_\QQ,$  we obtain
$$ \hat \varphi_\QQ(r_q)=(f_q^\QQ,1), $$
where $f_q^\QQ$ is the image of $f_q $ in $\QQ[\![x]\!].$ Consider the map
$$\Theta_\QQ : \QQ[\![x]\!] \longrightarrow H_2(\widehat{\LG}_\QQ,\QQ)$$
given by
$$ f \mapsto ((f,1) \wedge 1)\otimes \QQ.$$ Observe that this map is the composition of the map from Proposition \ref{prop_kernel_to_exterior} and the isomorphism from Proposition \ref{prop_homology_of_lamplighter}. Therefore the kernel of $\Theta_\QQ$ is countable. Set
$$ A:=\{ f_q^\QQ \mid q\in \{0,1\}^{\mathbb N} \} \subseteq \QQ[\![x]\!] .$$
Using that $f^\QQ_q=\sum_{i=3}^\infty  n_i(q) x^{i-1},$ where $n_{2i+1}(q)=q_i,$ we obtain that $A$ is uncountable. Using that the kernel of $\Theta_\QQ$ is countable, we obtain that its image $$\Theta_\QQ(A)=\{ ((f^\QQ_q,1) \wedge 1) \otimes \QQ \mid q\in \{0,1\}^{\mathbb N}\} \subseteq H_2(\widehat{\LG}_\QQ,\QQ)$$ is uncountable.
Finally, observe that any element $((f^\QQ_q,1) \wedge 1) \otimes \QQ$ of $\Theta_\QQ(A)$ has a preimage in $H_2(\hat F_\ZZ,\ZZ)$ given by $(r_q\wedge a)(s_q \wedge b),$ and then $\Theta_\QQ(A)$ lies in the image of $H_2(\hat F_\ZZ,\ZZ)\to H_2(\widehat{\LG}_\QQ,\QQ).$ This implies that the groups $H_2(\hat F_\QQ,\QQ)$ and $H_2(\hat F_\ZZ,\ZZ/p)\cong H_2(\hat F_\ZZ,\ZZ)\otimes \ZZ/p$ are uncountable and Theorems 1 and 2 follow.

\end{document}